\documentclass[BCOR=12mm, DIV=default,twoside, pagesize]{scrartcl} 
\usepackage{graphicx}
\usepackage{enumerate}
\usepackage{amsmath,amsthm,amssymb,mathabx,color,makeidx,pstricks,pst-node,pst-tree,graphicx,ifthen}
\usepackage{url}
\usepackage[plainpages=false]{hyperref} 
\usepackage{pdfpages}
\usepackage[all]{xy}

\definecolor{darkblue}{rgb}{0.3,0,1}
\definecolor{darkblue}{rgb}{0,0,0}
\usepackage{natbib}

\usepackage{tikz,pgfplots}
\usetikzlibrary{arrows,automata}

\newtheorem{proposition}{Proposition}[section] 

\newtheorem{thm}[proposition]{Theorem}
\newtheorem{lemma}[proposition]{Lemma}
\newtheorem{corollary}[proposition]{Corollary}
\newtheorem{remark}[proposition]{Remark}
\newtheorem{example}[proposition]{Example}

\usepackage[T1]{fontenc}
\usepackage{lmodern}

\def \N{\mathbb{N}}
\def \R{\mathbb{R}}

\def \NN{\mathbb{N}_0}

\def \pipe {\;\vert\;}

\newcommand{\umo}[3]{\underset{\text{\tiny{#1}}}{\stackrel{\mbox{\tiny{#3}}}{\vphantom{\geq}#2\vphantom{\geq}}}}
\newcommand{\omu}[3]{\umo{#3}{#2}{#1}}

\newcommand{\function}[3]{ #1:#2 \rightarrow #3}
\newcommand{\functionn}[4]{ #1:#2 \rightarrow #3, \; #4 \mapsto \,}

\newcommand{\haken}[1]{\underline{#1}{\raise -0.3ex\hbox{\vphantom{$#1$}\vrule}}}

\newcommand{\set}[1]{\left\{#1\right\}}
\newcommand{\sett}[2]{\left\{#1\;;\;#2\right\}}

\newcommand{\Wkt}[2]{#1\left(#2\right)}
\newcommand{\bedWkt}[3]{\Wkt{#1}{#2\pipe#3}}
\renewcommand\inf{\qopname\relax m{\vphantom{p}inf}}

\newcommand{\betrag}[1]{\;\left\vert#1\right\vert\;}
\newcommand{\betrags}[1]{\left\vert#1\right\vert}

\newcommand{\Lim}{\lim\limits}

\newcommand{\off}[1]{\iffalse#1\fi}

\newcommand{\Tau}{\mathcal{T}}

\newcommand{\MCStoppzeitenMenge}[1][]{\Tau\ifthenelse{\equal{}{#1}}{}{\Klammern{#1}}}
\newcommand{\MCStoppzeitenMengeDef}[1]{\sett{\tau\in\MCStoppzeitenMenge}{Z_\tau \in #1 \mbox{ on } \set{\tau<\infty}}}

\newcommand{\CondExp}[2]{\CondXY{#1}{#2}{E}}

\newcommand{\CondXY}[3]{\text{#3}\left(#1\middle|#2\right)}

\newcommand{\id}{1\hspace{-0,9ex}1}

\newcommand{\Eqref}[1]{\textbf{\eqref{#1}}}
\newcommand{\xyref}[3]{\hyperref[#1]{#2\ref*{#1}#3(page~\pageref{#1})}}

\newcommand{\TextRef}[2]{\hyperref[#1:#2]{#1~\ref*{#1:#2} (page~\pageref{#1:#2})}}
\newcommand{\TextRefD}[2]{\hyperref[#1:#2]{#1~\ref*{#1:#2} (page~\pageref{#1:#2})}}

\newcommand{\wnl}[3]{\ifthenelse{\equal{}{#2}}{}{#1#2#3}}
\newcommand{\fabist}[5]{%
	\ifmmode\text{ \wnl{}{#5}{ }}\else\wnl{}{#5}{ }\fi%
	\ifthenelse
	{\equal{#1}{0}\and\equal{#4}{\infty}\and{\equal{#3}{<}}}%
	{\ifmmode #2\in\NN\else$#2\in\NN$\fi}%
	{\ifthenelse
		{\equal{#1}{0}\and\equal{#4}{\infty}\and{\equal{#3}{\leq}}}%
		{\ifmmode #2\in\NN\cup\set{\infty}\else $#2\in\NN\cup\set{\infty}$\fi}%
		{\ifthenelse
			{\equal{#1}{1}\and\equal{#4}{\infty}\and{\equal{#3}{<}}}%
			{\ifmmode #2\in\N\else $#2\in\N$\fi}%
			{\ifthenelse
				{\equal{#4}{\infty}\and{\equal{#3}{\leq}}}%
				{\ifthenelse
					{\equal{#1}{1}}%
					{\ifmmode #2\in\N\cup\set{\infty} \else $#2\in\N\cup\set{\infty}$\fi}%
					{\ifmmode
						#2\in\NN\cup\set{\infty}\text{ with }#1\leq#2%
						\else%
						$#2\in\NN\cup\set{\infty}$ with $#1\leq#2$%
						\fi}%
				}
				{\ifthenelse
					{\equal{#1}{0}}%
					{\ifmmode #2\in\NN\text{ with }#2#3#4\else$#2\in\NN$ with $#2#3#4$\fi}%
					{\ifthenelse
						{\equal{#1}{1}}%
						{\ifmmode #2\in\N\text{ with }#2#3#4\else $#2\in\N$ with $#2#3#4$\fi}%
						{\ifmmode%
							#2\in\NN\text{ with }#1\leq#2%
							\ifthenelse{\equal{#4}{\infty}\and{\equal{#3}{<}}}{}{#3#4}%
							\else%
							$#2\in\NN$ with $#1\leq#2%
							\ifthenelse{\equal{#4}{\infty}\and{\equal{#3}{<}}}{}{#3#4}$%
							\fi%
						}%
					}%
				}%
			}%
		}%
	}%
}

\newcommand{\fabisss}[5]{#5#2
	\ifthenelse
	{\equal{#1}{0}\and\equal{#4}{\infty}\and{\equal{#3}{<}}}
	{\in\NN}
	{\ifthenelse
		{\equal{#4}{\infty}\and{\equal{#3}{\leq}}}
		{\text{ }\penalty 300\text{with }{#1\leq#2#3#4}}
		{\ifthenelse
			{\equal{#1}{0}}
			{\in\NN\text{ }\penalty 300\text{with }{#2#3#4}}
			{\in\NN\text{ }\penalty 300\text{with }{#1\leq#2#3#4}}
		}
	}
}

\newcommand{\Klammern}[1]{\left(#1\right)}

\newcommand{\XX}{X}

\newcommand{\tttt}[4]{#1_{\vphantom{N}#2\ifthenelse{\equal{#4}{}}{}{,#4}}^{\ifthenelse{\equal{#3}{}}{}{\ifthenelse{\equal{#3}{*}}{#3}{(#3)}}}}
\newcommand{\ttt}[4]{#1_{#4}^{\ifthenelse{\equal{#3}{}}{}{\ifthenelse{\equal{#3}{*}}{\vphantom{(#3)}#3}{(#3)}}}}

\newcommand{\privat}[1]{\iffalse#1\fi}

\newcommand{\integrala}[2]{\int\limits_{#1}#2\text{d}P_z}%

\begin{document}
\title{Flexible forward improvement iteration for infinite time horizon Markovian optimal stopping problems}
\author{S\"oren Christensen\thanks{Christian-Albrechts-Universit\"at, Mathematisches Seminar, Heinrich-Hecht-Platz 6, 24118 Kiel, Germany, E-mail: {christensen}@math.uni-kiel.de}, 
Albrecht Irle\thanks{Deceased on July 21, 2021}, and Julian Peter Lemburg}
\date{\today}
\maketitle

\begin{abstract}
In this paper, we propose an extension of the {forward improvement iteration} algorithm, originally introduced in \cite{Irl06} and recently reconsidered in \cite{miclo_villeneuve_2021}. The main new ingredient is a flexible window parameter describing the look-ahead distance in the improvement step. We consider the framework of a Markovian optimal stopping problem in discrete time with random discounting and infinite time horizon. We prove convergence and show that the additional flexibility may significantly reduce the runtime. 
\end{abstract}

\vspace{1mm}


\section{Introduction}
We consider a Markovian optimal stopping problem with random discounting of the form 
\[\sup_\tau\CondExp{\Klammern{\prod_{i=0}^{\tau-1}\alpha(Z_i)}g(Z_\tau)}{Z_0=z}.\]
As is well-known, an optimal stopping rule is (if it exists) given as a first entrance time of the underlying Markov chain $Z$ into the stopping set, see \cite{MR2374974}. The forward {forward improvement iteration} algorithm, see \cite{Irl06}, provides an upper approximation of this stopping set. More precisely, it consists of computing a decreasing sequence $B^0\supseteq B^1\supseteq B^2\supseteq \dots $ of subsets of the state space in the following way: $B^{n+1}$ is constructed from $B^n$ by taking out those states $z$ where using the first entrance time into $B^n$, that is $\inf\{n\geq 1:Z_n\in B^n\}$, gives more expected reward than stopping immediately. In this sense, 1-step look ahead rules are used in each iteration step to identify points which are not in the stopping set. The main finding reported in the literature are the convergence of this algorithm to the optimal value. For finite states spaces, the algorithm terminates in a finite number of steps. In the numerical examples it turns out that this number is often surprisingly low.

The aim of this paper is to study the following modification of the forward improvement iteration algorithm. Instead of considering 1-step look ahead rules, $\kappa$-steps look ahead rules are used in each iteration step, where $\kappa$ denotes a window parameter $\geq 1$ which may vary over time. This makes the analysis considerably more involved. The purpose of this paper is twofold. First, we prove convergence of this algorithm. Second, we will show in an example that using a flexible window parameter notably speeds up the algorithm.

\subsection{Some related literature}
Pricing of American and Bermudan style derivatives on a high dimensional system of underlyings is considered an enduring problem for the last decades and many different methods have been proposed. Prices for such high dimensional options are difficult, if not impossible, to compute by standard partial differential equation methods. Therefore, various Monte Carlo algorithms for pricing Bermudan options have been developed and described in the literature. For a detailed and general overview we refer to \cite{Gla03}.

However, a major part of these methods is based on a backward structure in time, which requires a finite time horizon. For infinite horizon problems, the number of suitable algorithms is more limited. Here, we just mention methods being based on stochastic approximation, see \cite{tsitsiklis1999optimal}, and linear programming, see \cite{cho2002linear,christensen2014method}.

Another stream of literature, which we follow here, is based in the classical method of policy iteration as proposed by Howard in \cite{howard:dynamic1960}.
Adapting this idea to problems of optimal stopping yields a method first used in \cite{Irle1980} for tackling best choice problems with a random population size,
and later called {forward improvement iteration}, see \cite{Irl06,Irl09} for a detailed analysis. The method has recently been reconsidered in \cite{miclo_villeneuve_2021} and, e.g., applied to restriced problems. See also \cite{presman2011new} for a related approach. 

Let us mention some further related literature. Bender, Kolodko and Schoenmarkers have developed a similar approach in \cite{KS06} and \cite{BS06} with some interesting new features, including a window parameter, for problems with a finite set of exercise dates. Their analysis is based on backward induction. In a continuous time framework, a related approach for one-dimensional diffusions is presented in \cite{sonin1999elimination}. Let us also mention that a related form of forward algorithm to find equilibrium stopping times in time-inconsistent problems is  studied in \cite[Section 3]{christensen2018finding}. Further references are given below.


\subsection{Structure of the paper}
We will introduce the framework and basic notions in Section \ref{sec:main_result} and then state the main result (Theorem \ref{thm:MainMarkov}). As a key ingredient for the proof of this result, we first analyze the one-step improvement of stopping rules in Section \ref{sec:one-step}, where two lemmata are moved to an appendix to enhance readability. These results are then used to prove Theorem \ref{thm:MainMarkov} in\ Section \ref{sec:proof}. For the implementation it is interesting to note that for finite state spaces the procedure can be described in terms of linear equations as detailed in Section \ref{sec:finite_state}. This is then used in Section \ref{sec:numerics} where an example is given.

\section{Setting and  main result}\label{sec:main_result}
Let $(Z_n)_{n\in\NN}$ be a time-homogeneous Markov chain on a state space $S$ on some probability space with filtration $\mathcal{A}_n,\,n\in\N_0$. To avoid technicalities, we assume $S$ to be countable (but not necessarily finite). Let furthermore
$\function{g}{S}{\R}$ and
$\function{\alpha}{S}{[0,1]}$. We write $\mathcal T$ for the set of all stopping rules and look at the problem of optimal stopping for the pay-off
\[
\XX_n:=\Klammern{\prod_{i=0}^{n-1}\alpha(Z_i)}g(Z_n)
\fabist{0}{n}{<}{\infty}{for all},\;\;
\XX_{\infty}:=\limsup_{n}{\XX_n}
.\]
We write for simpler notation formally
\[
\Klammern{\prod_{i=0}^{\infty}\alpha(Z_i)}g(Z_\infty)=X_\infty
\]
so that we have
\[
X_\tau = \Klammern{\prod_{i=0}^{\tau-1}\alpha(Z_i)}g(Z_\tau)\mbox{ for all stopping rules }\tau.
\]
Assume as usual
\begin{align*}
\CondExp{\XX_\tau}{Z_0=z} \text{ exists for all }\tau\in\MCStoppzeitenMenge\text{ and }z\in S
.\end{align*}
Furthermore, we assume that the integrability condition
\begin{align*}
\CondExp{X_{\lim\limits_{n\to\infty}\sigma_n}}{Z_0=z}
=
\lim\limits_{n\to\infty}\CondExp{X_{\sigma_n}}{Z_0=z}\nonumber
\end{align*}
holds true for all $z\in S$ and all non-decreasing sequences $(\sigma_n)_{n\in\NN}$ of stopping rules. It is easily seen that the previous assumptions are fulfilled under the standard conditions
\begin{align*}
\text{$\CondExp{\sup_{n\in\N}\betrags{X_n}}{Z_0=z}<\infty$ and $X_{\infty}=\lim_{n\to\infty}X_n$}.
\end{align*}

In this framework, we consider the problem of optimal stopping for $X$, i.e., the problem of maximizing, for all\ $z\in S$, the expectation
\[\CondExp{\XX_\tau}{Z_0=z}\]
over all stopping rules $\tau\in \mathcal T$. Here, we even consider a more general constraint problem of stopping, namely the maximization over $\tau$ in 
\[
\MCStoppzeitenMenge[B]:=\MCStoppzeitenMengeDef{B}
\]
for some given set $B\subseteq S$. We remark that this class of problems naturally arises in problems of optimal  control that are based on stopping, see -- in a continuous time framework -- e.g. \cite{christensen2020solution}.

\begin{remark}
	In the recent article  \cite{miclo_villeneuve_2021}, a (time-homogeneous) Markov chain in continuous time $(\widehat{Z}_t)_{t\in[0,\infty[}$ 
	operating on the discrete state space $S$ is considered.
	Assuming no instantaneous jumps (i.\,e.\  between two jump times there is always a positive amount of time) and no absorption, the above Markov chain has a pure jump structure, it is a special Markov jump process and we can describe it up to explosion time completely by a sequence of jump times and a sequence of states visited with adequate parameters, see \cite[p. 39-40]{Asm03}. When focusing on the problem of optimal stopping of the above continuous time Markov chain we can reduce this problem to one studied in the present paper.
\end{remark}

For all $B\subseteq S$ and $\sigma\in\Tau$ define
\begin{align*}
\tau_{\sigma}(B)
:=
\inf\sett{k\geq \sigma}{Z_k\in B}
,\end{align*}
the time of first entrance in $B$ at or after time $\sigma$.
Furthermore, write for all $B\subseteq S$ and $k\in\N$
\begin{eqnarray*}
	B^{*k}
	:=
	\sett{z\in B}{g(z)\geq{\CondExp{\XX_{\tau_{l}(B)}}{Z_0=z}}\mbox{ for all }l=1,...,k} 
	\end{eqnarray*}
for the set of all states such that looking $l=1,...,k$ steps ahead and then waiting until reaching $B$ does not yield more expected reward than immediate stopping. Note that for $B=S$, the first entrance time into $S^{*k}$ is called $k$-step look ahead rule and had been studied extensively in the early days of optimal stopping, see \cite[Chapter 5]{Ferguson} for an overview with many examples. The case $k=1$ corresponds to the myopic rule, which is optimal for problems of monotone type under natural assumptions, see \cite{CRS1971} for an early treatment and \cite{CI19} for an overview on more recent developments. 

For each iteration in the algorithm, we now want to choose some window parameter $k$. $k$ may be chosen differently in each step, so we  therefore fix a window parameter function
$$
\function{\kappa}{\N}{\N}
$$ 

To make the procedure precise, for a given $B^0\subseteq S$, we define recursively
\begin{eqnarray*}
B^k:=(B^{k-1})^{*\kappa(k)}\fabist{1}{k}{<}{\infty}{for all},
\quad
F := \bigcap_{k\in\NN}B^k
.\end{eqnarray*}


Our main result is the following:

\begin{thm}\label{thm:MainMarkov}
Let 
 $B^0\subseteq S$.
%
Then $\tau_0(F)$ is optimal in $\MCStoppzeitenMenge[B^0]$ and moreover
\begin{gather}
\tau_0(B^{k})
\leq
\tau_0(B^{k+1})
\fabist{0}{k}{<}{\infty}{for all},\label{1.6.3.A}
\\	
\lim\limits_{k\to\infty}\tau_0(B^{k})=\tau_{0}(F),
\label{1.6.3.B}
\\	
\CondExp{X_{\tau_0(B^{k})}}{Z_0=z}
\leq
\CondExp{X_{\tau_0(B^{k+1})}}{Z_0=z}
\fabist{0}{k}{<}{\infty}{for all}
\text{ for all }z\in S.
\label{1.6.3.C}
\end{gather}
\end{thm}

We will give the proof in Section \ref{sec:proof} below.

\section{One-step improvement}\label{sec:one-step}
Before proving the main result, we first study a one-step policy improvement of a given stopping rule inspired by the algorithm described before.

While running numerical examples (see Section \ref{sec:numerics} below) it appears that 
using more general $D\subseteq\N$ (usually with $1\in D$)
for the improvement step 
instead of an initial sequence of $\N$
can be of interest.
Hence let us 
for all $B\subseteq S$ and $D\subseteq\N$ generalize the notion from above by introducing the notation
\begin{eqnarray*}
	B^{*D}
	:=
	\sett{z\in B}{g(z)\geq{\CondExp{\XX_{\tau_{k}(B)}}{Z_0=z}}\mbox{ for all }k\in\ D} 
	.\end{eqnarray*}

\newcommand{\sigmastern}{\tau_\sigma\left(B^{*D}\right)}
\renewcommand{\sigmastern}{\tau_\sigma(B^{*D})}

For all $B\subseteq S$, $D\subseteq\N$, $\sigma\in\MCStoppzeitenMenge[B]$ and
$
\rho\in\MCStoppzeitenMenge[B]
\text{ with }
\sigma\leq\rho\leq\sigmastern
$
define the \emph{improved} stopping rule $\widehat\rho$ via
\begin{gather}
\widehat\rho
:=
\begin{cases}
\rho,&\mbox{ on }\{\rho=\sigmastern\},\\
\tau_{n+j}(B)\wedge\sigmastern,&\mbox{ on }\{\rho=n,\inf\sett{i\in D}{Z_n\not\in B^{*D_{\leq i}}}=j\},
\end{cases}
\label{rhodefinition}
\end{gather}
where we use the notation $D_{\leq i}=\{l\in\ D:l\leq j\}$.
We now show that $\widehat{\rho}$ is indeed an improvement of $\rho$ in the following sense:
	\begin{proposition}\label{MainMarkov}
	Let $B\subseteq S$, $D\subseteq\N$, $\sigma\in\MCStoppzeitenMenge[B]$ and
	\begin{eqnarray*}
	\rho\in\MCStoppzeitenMenge[B]
	\text{ with }
	\sigma\leq\rho\leq\sigmastern
	.\end{eqnarray*}
	\begin{enumerate}
		\item 
	It holds that
	\begin{gather}
	\widehat\rho\in\MCStoppzeitenMenge[B]
	\text{ with }\sigma\leq\widehat\rho\leq\sigmastern
	\label{rhoforderung1}
	,\\	\rho\leq\widehat\rho
	\label{rhoforderung2}
	,\\	\rho+1\leq\widehat\rho
	\text{ on }\set{\rho<\sigmastern}
	\label{rhoforderung3}
	.\end{gather}%
\item
If $D=\{l\in \N:l\leq k\}$ for some $k\in\N\cup\set\infty$, then 
\begin{gather}
\CondExp{X_{\rho}}{Z_0=z}
\leq
\CondExp{X_{\widehat\rho}}{Z_0=z}
\text{ for all }z\in S\label{rhoforderung4}
,\end{gather}
and
\begin{eqnarray}
\CondExp{X_{\sigma}}{Z_0=z}
\leq
\CondExp{X_{\sigmastern}}{Z_0=z}
\text{ for all }z\in S\label{rhoforderung5}
.\end{eqnarray}
\end{enumerate}
\end{proposition}

\begin{proof}
	We start by proving \Eqref{rhoforderung1}, \Eqref{rhoforderung2}, \Eqref{rhoforderung3}.
	We have
	\[
	\rho=\widehat\rho=\sigmastern
	\text{ on }
	\set{\rho=\sigmastern}=\set{\rho=\infty}\cup\set{Z_{\rho}\in B^{*D}}
	.\]
	Furthermore, $B^{*\emptyset}=B$ and
	\begin{align*}
	B^{*D}=\bigcap\limits_{j\in D}B^{*D_{\leq j}}.
	\end{align*}
	Hence
	\[
	\set{Z_n\not\in B^{*D}}
	=
	\bigcup_{j\in D}\set{j=\inf\sett{i\in D}{Z_n\not\in B^{*D_{\leq i}}}}
	\fabist{0}{n}{<}{\infty}{for all}
	.\]
	So \fabist{0}{n}{<}{\infty}{for all} and $j\in D$ we have that on the event $\set{\rho=n}\cap\set{j=\inf\sett{i\in D}{Z_n\not\in B^{*D_{\leq i}}}}$ it holds that
	\begin{align*}
	\rho<\rho+1=n+1\leq\overbrace{\Klammern{\tau_{n+j}(B)\wedge\sigmastern}}^{=\widehat\rho}\leq\sigmastern
	.\end{align*}
	Hence
	\begin{align*}
	\set{\rho<\sigmastern}
	=
	\set{\rho<\infty}\cap\set{Z_{\rho}\not\in B^{*D}}
	=
	\bigcup_{n=0}^{\infty}\set{\rho=n}\cap\set{Z_n\not\in B^{*D}}
	.\end{align*}
	Now we come to the proof of \eqref{rhoforderung4}:
	Let $z\in S$ and define 
	\[
	A
	:=
	\set{\rho<\sigmastern}
=
	\set{\rho<\infty}\cap\set{Z_{\rho}\not\in B^{*D}}
	.\]
	We have
	\[
	A^c
	=
	\set{\rho=\infty}\cup\set{Z_{\rho}\in B^{*D}}
	=
	\set{\rho=\sigmastern}
	\subseteq
	\set{\rho=\widehat\rho}
	\]
	and thus obviously
	\[
	\CondExp{\id_{A^c}X_{\rho}}{Z_0=z}
	\leq
	\CondExp{\id_{A^c}X_{\widehat\rho}}{Z_0=z}
	.\]
	Define \fabist{0}{n}{<}{\infty}{for all} and $j\in D$
	\[
	A^{n,j} := \set{\rho=n}\cap\set{j=\inf\sett{i\in D}{Z_n\not\in B^{*D_{\leq i}}}}
	.\]
Hence
	\[
	A=\biguplus\limits_{n=0}^{\infty}\biguplus\limits_{j\in D}^{}A^{n,j}
	.\]
	\fabist{0}{n}{<}{\infty}{For all} and $j\in D$ we have by Lemma \ref{PartialImproving}
	\begin{eqnarray*}
	\CondExp{\id_{A^{n,j}}X_{\rho}}{Z_0=z}
	=
	\CondExp{\id_{A^{n,j}}X_{n}}{Z_0=z}
	\leq
	\CondExp{\id_{A^{n,j}}X_{\tau_{n+j}(B)}}{Z_0=z}
	.\end{eqnarray*}
	Due to dominated convergence it suffices to show \fabist{0}{n}{<}{\infty}{for all} and $j\in D$ the inequality
	\begin{align*}
	\CondExp{\id_{A^{n,j}}X_{\tau_{n+j}(B)}}{Z_0=z}
	\leq
	\CondExp{\id_{A^{n,j}}X_{\widehat\rho}}{Z_0=z}
	.\end{align*}

	To this end, define
	\[
	E:=\set{\tau_{n+j}(B)\leq\sigmastern}
	.\] 
	We have
	\begin{eqnarray*}
		\tau_{n+j}(B)
		=
		\tau_{n+j}(B)\wedge\sigmastern
		=
		\widehat\rho
		\text{ on }
		A^{n,j}\cap E
		,\end{eqnarray*}
	hence
	\begin{eqnarray*}
	\CondExp{\id_{A^{n,j}\cap E}X_{\tau_{n+j}(B)}}{Z_0=z}
	\leq
	\CondExp{\id_{A^{n,j}\cap E}X_{\widehat\rho}}{Z_0=z}
	.\end{eqnarray*}
	
	\fabist{1}{m}{<}{j}{For all} define
	\begin{eqnarray*}
	A_m:=A^{n,j}\cap\set{Z_{n+m}\in B^{*D}}\cap\bigcap_{l=1}^{m-1}\set{Z_{n+l}\not\in B^{*D}}
	.\end{eqnarray*}
	
	\fabist{1}{m}{<}{j}{For all} and $j-m\in D$
	it holds
	by {Lemma} \ref{PartialImprovingII}
	\begin{align*}
	\integrala{A_m}{\XX_{\tau_{n+j}(B)}}
\leq
	\integrala{A_m}{\XX_{n+m}}
=	\integrala{A_m}{\XX_{\sigmastern}}
=
	\integrala{A_m}{\XX_{\widehat\rho}}
	.\end{align*}
	
	We will prove next
	\begin{align*}
	\Klammern{A^{n,j}\cap E^c}
	=
	\biguplus_{m=1}^{j-1}A_m
	.\end{align*}
	Due to linearity of expectation this finishes the proof.

	\fabist{1}{m,l}{<}{j}{For all}, $m\neq l$ we have
	\[
	A_m\cap A_l
	\subseteq
	\set{Z_{n+m}\in B^{*D}}\cap\set{Z_{n+m}\not\in B^{*D}}
	=
	\emptyset
	\text{, hence }
	A_m\cap A_l = \emptyset
	.\]
	\fabist{1}{m}{<}{j}{For all} we have due to $n+m<n+j\leq\tau_{n+j}(B)$ and the definition of $A_m$
	\[
	A_m
\subseteq	\set{\sigmastern=n+m}
	\subseteq
	\set{\sigmastern<\tau_{n+j}(B)}
=
	E^c
	.\]
	
	Let $\omega\in A^{n,j}\cap E^c$.
	\\\fabist{0}{l}{<}{\infty}{For all} 
	with $n+j\leq l<\tau_{n+j}(B)(\omega)=\inf\sett{l\geq n+j}{Z_l(\omega)\in B}$
	we obviously have $Z_l(\omega)\not\in B$ and due to $B^{*D}\subseteq B$ we have $Z_l(\omega)\not\in B^{*D}$.
	\\Since we consider
	\[
	\omega
	\in
	A^{n,j}\cap E^c
	\subseteq
	\set{\sigmastern<\tau_{n+j}(B)}\cap\set{\rho=n\vphantom{\tau_j^{(D}}}
	,\]
	hence 
	\[
	\tau_{n+j}(B)(\omega)
	>
	\sigmastern(\omega)
	=
	\inf\sett{l\geq\sigma(\omega)}{Z_l(\omega)\in B^{*D}}
	,\]
	we obtain $n+j>\sigmastern(\omega)$ and  $\sigmastern(\omega)\geq\rho(\omega)=n$.
	By the definition of $A^{n,j}$ we have $Z_n(\omega)\not\in B^{*D_{\leq j}}\supseteq B^{*D}$, 
	hence $\sigmastern(\omega)\neq n$.
	So we have $n<\sigmastern(\omega)<n+j$, 
	thus it follows $\omega\in A_{\left[\sigmastern(\omega)-n\right]}$ by definition.

	It remains to prove \eqref{rhoforderung5}.
	Define $\sigma_0:=\sigma$ and inductively
	\[
	\sigma_{k+1}:=\widehat\sigma_{k} \fabist{0}{k}{<}{\infty}{for all}
	.\]
	
A  simple induction yields
	\begin{eqnarray*}
	\sigma_k\leq\sigma_{k+1}\leq\sigmastern
	\fabist{0}{k}{<}{\infty}{for all}
	.\end{eqnarray*}
	
Therefore,
	\[
	\lim_{k\to\infty}{\sigma_k}\leq\sigmastern
	\]
	and by \eqref{rhoforderung3}
	\[
	\sigma_k+1\leq\sigma_{k+1}
	\text{ on }\set{\sigma_k<\sigmastern}
	\fabist{1}{k}{<}{\infty}{for all}
	,\]
	hence  we have
	\[
	\sigma_0\leq\sigma_1\leq\sigma_2\leq...\leq\lim_{k\to\infty}\sigma_k=\sigmastern
	.\]
	
	Let $z\in S$.
	We have \fabist{0}{k}{<}{\infty}{for all} by \eqref{rhoforderung4}
	\begin{align*}
	\CondExp{X_{\sigma_{k}}}{Z_0=z}
	&\leq
	\CondExp{X_{\sigma_{k+1}}}{Z_0=z}\label{eq.1.5.8.7}
	&&
	\text{}
	.\end{align*}
	We infer using our standing assumption
	\begin{align*}
	\CondExp{X_{\sigma}}{Z_0=z}
	&=
	\CondExp{X_{\sigma_0}}{Z_0=z}
	\leq
	\lim\limits_{k\to\infty}\CondExp{X_{\sigma_{k}}}{Z_0=z}
	\\	&=
	\CondExp{X_{\lim\limits_{k\to\infty}\sigma_{k}}}{Z_0=z}
		=
	\CondExp{X_{\sigmastern}}{Z_0=z}
	\end{align*}
\end{proof}

\begin{example}
	Using the stopping rule
	\[
	\widetilde\rho
	:=
	\begin{cases}
	\rho,&\mbox{ on }\{\rho=\sigmastern\},\\
	\tau_{n+j}(B),&\mbox{ on }\{\rho=n,\inf\sett{i\in D}{Z_n\not\in B^{*D_{\leq i}}}=j\},
	\end{cases}
	\]
	instead of $\widehat \rho$ in \eqref{rhodefinition} looks natural at first glance. However, then the properties \eqref{rhoforderung1} and \eqref{rhoforderung5} may not hold.
	We can see this by the following example:
	Consider a Markov chain given by
	\[
	\entrymodifiers={++[o][F-]} 
	\SelectTips{cm}{} 
	\xymatrix @-1pc {
		c \ar[drr]_{1}
		& *\txt{}
		&a \ar[rr]^{1/3} \ar[ll]_{1/3} \ar[d]^{1/3}
		&*\txt{}
		&{d} \ar[rr]^1
		&*\txt{}
		& e \ar@(u,r)[]^{1} 
		\\
		*\txt{}&*\txt{}& b \ar@(r,d)[]^{1} }
	\]
	
	
	and the rewards, i.e. the function $g$
	\[
	\begin{array}{r|l} \text{state}&\text{payment}\\\hline a & 3\\b&4\\c&1.5\\d&2.5\\e&2\end{array}
	\]
	and we will use $D=\set{1,2}$ with $\alpha\equiv1$.
	We initialize the algorithm with 
	\[
	B=\set{a,b,c,d,e}.
	\] 
	Hence we have 
	\[
	B^{*\set 1}=\set{a,b,d,e}
	\]
	and
	\[
	B^{*\set{1,2}}=\set{b,d,e}
	.\]
	Now we determine $\widetilde{\rho}$ for $\rho:=\tau_\sigma(B^{*\set 1})$.
	On $\set{Z_0\neq a}$ we have $\widetilde{\rho}=\tau_\sigma(B^{*\set{1,2}})$.\\
	If $Z_0=a$,
	then $\rho=n=0$, hence 
	\[
	j:=\inf\sett{i\in D}{Z_n\not\in B^{*D_{\leq i}}}=\inf\sett{i\in D}{a\not\in B^{*D_{\leq i}}}=2
	\]
	and thus on $\set{Z_0= a}$
	\begin{align*}
	\widetilde{\rho}
	=\tau_{n+j}(B)
	=\tau_{0+2}(B)
	=\tau_{2}(B)
	&=\inf\sett{n\geq2}{Z_n\in\set{b,e}}
	\\	&>\inf\sett{n\geq0}{Z_n\in\set{b,d,e}}
	=\tau_\sigma(B^{*\set{1,2}})
	.\end{align*}
	Hence it is necessary to take the minimum with $\sigma^{*D}$ for a sufficient improvement, in the sense of having $\widetilde{\rho}$ always smaller than $\sigma^{*D}$ and not giving a smaller expected revenue.
\end{example}

\section{Proof of Theorem \ref{thm:MainMarkov}}\label{sec:proof}
\begin{proof}
	Let \fabist{0}{k}{<}{\infty}{}. 
	We have 
	\[
	B^{k+1}=(B^k)^{*\kappa(k)}\omu{}{\supseteq}{}B^k
	\]
	and 
\[
\tau_0(B^k)
\leq\tau_{\tau_0(B^k)}((B^k)^{*\kappa(k)})
=\tau_{\tau_0(B^k)}(B^{k+1})
=\tau_0(B^{k+1})
.\]
		Hence we can infer \eqref{1.6.3.A} and \eqref{1.6.3.B}, see Proposition \ref{MainMarkov}.
\eqref{1.6.3.C} also immediately follows using Proposition \ref{MainMarkov}.

Now, we come to the proof of the optimality.\\
(1) At first we will show that for any $\rho\in\MCStoppzeitenMenge[B^0]$ there exists
\begin{center}
	\(
	\rho^{\infty}\in\MCStoppzeitenMenge[F]
	\)
	with $\rho^{\infty}\geq\rho$ 
	and $\CondExp{X_{\rho}}{Z_0=z}\leq\CondExp{X_{\rho^{\infty}}}{Z_0=z}$ for all $z\in S$.
\end{center}
Let $\rho\in\MCStoppzeitenMenge[B^0]$. Define 
\begin{eqnarray*}
	\rho^k
	&:=&
	\inf\sett{n\geq \rho}{Z_n\in B^k}\fabist{0}{k}{<}{\infty}{for all},
	\\	\rho^{\infty}
	&:=&
	\inf\sett{n\geq\rho}{Z_n\in F}
	.\end{eqnarray*}
Since $(B^k)_{k\in\NN}$ is a non-increasing sequence of sets with limit $F$,
$(\rho^k)_{k\in\NN}$ is a non-decreasing sequence of stopping rules with limit $\rho^{\infty}$.
Since $\rho\in\MCStoppzeitenMenge[B^0]$, it furthermore holds that $\rho=\rho^0$.
Using {Proposition} \ref{MainMarkov} and our standing assumption, it follows
\begin{align*}
\CondExp{X_{\rho}}{Z_0=z}	
&~\leq~
\Lim_{k\to\infty}\CondExp{X_{\rho^{k}}}{Z_0=z}	
\\	&
\leq
\CondExp{\!X_{\Lim_{k\to\infty}\rho^{k}}}{Z_0=z\!}	
\!=
\CondExp{X_{\rho^{\infty}}}{Z_0=z}	
\text{ for all }z\in S.\!\!\!\!\!\!
\end{align*}
(2)
 Let $\rho,\tau\in\MCStoppzeitenMenge[F]$ such that $\rho\leq\tau$.
We will now show 
\[
\CondExp{\XX_{\tau}}{Z_0=z}
\leq
\CondExp{\XX_{\rho}}{Z_0=z}
\text{ for all }z\in S
.\]
Define
\[
\rho_k
:=
\rho\id_{\set{\rho=\tau}}
+
\id_{\set{\rho<\tau}}\sum_{n=0}^{\infty}\id_{\set{\rho=n}}\tau_{n+1}(B^k)
\fabist{0}{k}{<}{\infty}{for all}
,\]
\[
\rho_{\infty}
:=
\rho\id_{\set{\rho=\tau}}
+
\id_{\set{\rho<\tau}}\sum_{n=0}^{\infty}\id_{\set{\rho=n}}\tau_{n+1}(F)
.\]
Then we have 
$\rho\leq\rho_k\leq\rho_{k+1}\leq\rho_{\infty}\leq\tau$ \fabist{0}{k}{<}{\infty}{for all}
and 
$\Lim_{k\to\infty}\rho_k=\rho_{\infty}$.
Furthermore we have 
\[
\rho+1=n+1\leq\tau_{n+1}(F)=\rho_{\infty}
\text{ on }
\set{\rho<\tau}\cap\set{\rho=n}
\fabist{0}{n}{<}{\infty}{for all}
,\]
hence
\begin{eqnarray}
\rho+1\leq\rho_{\infty}
\text{ on }
\biguplus_{n=0}^{\infty}\set{\rho=n}\cup\set{\rho<\tau}
=
\set{\rho\neq\tau}
\label{rhoplus1}
.\end{eqnarray}
Let $(Z_n')_{n\in\NN}$ be an independent copy of $(Z_n)_{n\in\NN}$ and 
\fabist{0}{p}{<}{\infty}{for all} and $B\subseteq S$ let $\tau_p'(B)=\inf\sett{q\geq p}{Z_q'\in B}$.
Let \fabist{0}{n}{<}{\infty}{} and \fabist{0}{k}{<}{\infty}{}.
Since $\rho\in\MCStoppzeitenMenge[F]$ we have $Z_{\rho}\in F$ on $\set{\rho<\infty}$ and thus
\begin{eqnarray*}
	\set{\rho=n}\cap\set{\rho<\tau}
	&\subseteq&
	\set{Z_n\in F}
	\\&\subseteq&
	\set{Z_n\in (B^k)^{*\kappa(k))}}
	\\&\subseteq&
	\set{g(Z_n)\geq\CondExp{\Klammern{\prod_{i=0}^{\tau'_{1}(B^k)-1}\alpha(Z_i')}g(Z_{\tau'_{n+1}(B^k)}')}{Z_0'=Z_n}}
	,\end{eqnarray*}
hence we have for all $z\in S$
\begin{eqnarray*}
	\integrala{\set{\rho=n<\tau}}{\XX_{\rho_k}}
	&=&
	\integrala{\set{\rho=n<\tau}}{\Klammern{\prod_{i=0}^{\tau_{n+1}(B^k)-1}\alpha(Z_i)}g(Z_{\tau_{n+1}(B^k)})}
	\\&=&
	\integrala{\set{\rho=n<\tau}}{
		\CondExp{\Klammern{\prod_{i=0}^{\tau'_{1}(B^k)-1}\alpha(Z_i')}g(Z_{\tau'_{n+1}(B^k)}')}{Z_0'=Z_n}
		\prod_{i=0}^{n-1}\alpha(Z_i)}
	\\&\leq&
	\integrala{\set{\rho=n<\tau}}{g(Z_n)\prod_{i=0}^{n-1}\alpha(Z_i)}
	\\&=&
	\integrala{\set{\rho=n<\tau}}{\XX_{\rho}}
	.\end{eqnarray*}
By dominated convergence 
follows
\begin{eqnarray*}
	\CondExp{\XX_{\rho_k}}{Z_0=z}
	\leq
	\CondExp{\XX_{\rho}}{Z_0=z}
	\text{ for all }z\in S
	.\end{eqnarray*}
Letting $k\to\infty$ we obtain 
\begin{eqnarray*}
\CondExp{\XX_{\rho_{\infty}}}{Z_0=z}
\leq
\CondExp{\XX_{\rho}}{Z_0=z}
\text{ for all }z\in S
.\end{eqnarray*}
Define $\rho^0=\rho$ and $\rho^k=(\rho^{k-1})_{\infty}$ \fabist{1}{k}{<}{\infty}{for all}.
Then obviously $\rho\leq\rho^k\leq\rho^{k+1}\leq\tau$, and $\rho^k\in\MCStoppzeitenMenge[F] \fabist{1}{k}{<}{\infty}{for all}$.
Hence we have by \eqref{rhoplus1} $\lim\limits_{k\to\infty}\rho^k=\tau$.\\
It follows for all $z\in S$ using our standing assumption
\[
\CondExp{\XX_{\tau}}{Z_0=z}
=
\CondExp{\XX_{\lim\limits_{k\to\infty}\rho^k}}{Z_0=z}
=
\lim\limits_{k\to\infty}\CondExp{\XX_{\rho^k}}{Z_0=z}
\leq 
\CondExp{\XX_{\rho}}{Z_0=z}
.\]
(3)
Now, we prove that $\tau(F)$ is optimal. 
Let 
$\sigma\in\MCStoppzeitenMenge[B^0]$.
By step 1, there exists $\tau\in\MCStoppzeitenMenge[F]$ with $\tau\geq \sigma$ such that
\[
\CondExp{\XX_{\sigma}}{Z_0=z}
\leq
\CondExp{\XX_{\tau}}{Z_0=z}
\text{ for all }z\in S
.\]
Obviously we have 
 $\tau_0(F)\leq\tau$ due to $\tau\in\MCStoppzeitenMenge[F]$.
Hence it follows by step 2 that
\[
\CondExp{\XX_{\sigma}}{Z_0=z}
\leq
\CondExp{\XX_{\tau}}{Z_0=z}
\leq
\CondExp{\XX_{\tau_0(F)}}{Z_0=z}
\text{ for all }z\in S
.\]
\end{proof}


\section{Finite state space}\label{sec:finite_state}
We now assume additionally that we have a finite state space. 
The algorithm requires the calculation of many conditional expectations. 
We show in this section that these expectations can be found by solving linear equations. 

Let $\Pi$ be the transition matrix
and define $\Psi$ by $\Psi_{z,y}:=\alpha(z)\cdot\Pi_{z,y}$ for all $z,y\in S$.

Write for all $p\in\N$ and all $B\subseteq S$
\[
\functionn{h'_{B,p}}{S}{\R}{z}
\CondExp{X_{\tau_p(B)}}{Z_0=z}
=\CondExp{g\Klammern{Z_{\tau_p(B)}}\prod_{j=0}^{\tau_p(B)-1}\alpha(Z_i)}{Z_0=z}
.\]
\begin{proposition}
For all $B\subseteq S$, $p\in\N$ and $z\in S$ we have
\[
h'_{B,p}(z)=\sum_{y\in S}\Klammern{\Psi^p}_{z,y}h'_{B,0}(y)
.\]
\end{proposition}
\begin{proof}
	Let $B\subseteq S$. 
	\\We have for all $p\in\N$ and $z\in S$ using the Markov property
	\begin{eqnarray*}
		h'_{B,p}(z)
		&=&
		\CondExp{g\Klammern{Z_{\tau_p(B)}}\prod_{j=0}^{\tau_p(B)-1}\alpha(Z_i)}{Z_0=z}
		\\	&=&
		\sum_{y\in S}\alpha(z)\bedWkt{P}{Z_1=y}{Z_0=z}
		\CondExp{g\Klammern{Z_{\tau_{p-1}(B)}}\prod_{j=0}^{\tau_{p-1}(B)-1}\alpha(Z_i)}{Z_0=y}
		\\	&=&\sum_{y\in S}\Psi_{z,y}
		\CondExp{g\Klammern{Z_{\tau_{p-1}(B)}}\prod_{j=0}^{\tau_{p-1}(B)-1}\alpha(Z_i)}{Z_0=y}
		\\	&=&\sum_{y\in S}\Psi_{z,y}
		h'_{B,p-1}(y)
		.\end{eqnarray*}
	Hence we have 
	\[	
	h'_{B,p}=\Psi\cdot h'_{B,p-1}
	.\]
	By induction we have for all $p\in\N$
	\[	
	h'_{B,p}=\Psi^p\cdot h'_{B,0}.
	\]
\end{proof}

We immediately obtain:

\begin{corollary}
	We have for all
$B\subseteq S$ and $D\subseteq\N$
\[
B^{*D}=\sett{z\in B}{g(z)\geq \sup_{p\in D}\sum_{y\in S}\Klammern{\Psi^p}_{z,y}h'_{B,0}(y)}
.\]
\end{corollary}


We may now rewrite the approach using matrices: 
Define the vector $b$ by 
\[
b_i:=
\begin{cases}
1&i\in B,
\\	0&i\not\in B,
\end{cases}
~~~\text{ for all }i\in S
,\]
the vector $d$ by
\[
d_i:=b_i\cdot g(i)\text{ for all }i\in S
\]
and the matrix $A$ by 
\[
A_{i,j}:=\delta_{i,j}-(1-b_i)\Psi_{i,j}
\text{ for all }i,j\in S
.\]

\begin{corollary}
Let $B\subseteq S$ and assume that {for all }$z\in S\ $
\begin{eqnarray}
	{\alpha(z)<1}
	\text{ or }
	{\bedWkt{P}{\tau_0(B)<\infty}{Z_0=z}=1}
\label{bedmc}
.\end{eqnarray}
Then $h'_{B,0}$ is the unique solution of the system of linear equations $Ah=d$.
\end{corollary}

\begin{proof}
	For all $h\in\R^S$ we have by definition
	\begin{eqnarray*}
		d=Ah \Longleftrightarrow \forall z\in B\;\;h(z)=g(z)\wedge\forall z\not\in B\;\; h(z)=\alpha(z) \sum_{y\in S}\bedWkt{P}{Z_1=y}{Z_0=z}h(y).
	\end{eqnarray*}
	For all $z\in B$ we have $Z_0=z\in B$ and thus 
	$\tau_0(B)=0$ and 
	$\prod\limits_{j=0}^{\tau_0(B)-1}\alpha(Z_i)=1$,  hence 
	\begin{eqnarray*}
		h'_{B,0}(z)
		=\CondExp{g\Klammern{Z_{\tau_0(B)}}\prod_{j=0}^{\tau_0(B)-1}\alpha(Z_i)}{Z_0=z}
		=\CondExp{g\Klammern{Z_{0}}\cdot 1}{Z_0=z}
		=g\Klammern{z}
	\end{eqnarray*}
	and for all $z\not\in B$ we have using the Markov property
	\begin{eqnarray*}
		h'_{B,0}(z)
		&=&
		\alpha(z)\sum_{y\in S}\bedWkt{P}{Z_1=y}{Z_0=z}
		\CondExp{g\Klammern{Z_{\tau_0(B)}}\prod_{j=0}^{\tau_0(B)-1}\alpha(Z_i)}{Z_0=y}
		\\	&=&
		\alpha(z)\sum_{y\in S}\bedWkt{P}{Z_1=y}{Z_0=z}h'(y)
		.\end{eqnarray*}
	So we have $d=Ah'_{B,0}$.

We now prove uniqueness by doing so for the corresponding homogeneous system of linear equations: 
	Hence consider some $h\in\R^S$ with $0=Ah$.
	It is obvious that for all $y\in B$ we have $h(y)=0$.
	For all $z\not\in B$ 
	\begin{eqnarray*}
		h(z)
		&=&
		\alpha(z) \sum_{y\in S}\bedWkt{P}{Z_1=y}{Z_0=z}h(y)
		\\	&=&
		\sum_{y\not\in B}\alpha(z)\bedWkt{P}{Z_1=y}{Z_0=z}h(y)
		\\	&=&
		\sum_{y\not\in B}\Psi_{z,y}h(y) 
		.\end{eqnarray*}
	Moreover,
	\begin{eqnarray*}
		\betrag{h(z)}
		=
		\betrag{\sum_{y\not\in B}\Psi_{z,y}h(y)}
		&\leq&
		\sum_{y\not\in B}\Psi_{z,y}\betrag{h(y)}
		\\	&\leq&
		\max\sett{\betrag{h(x)}}{x\not\in B}\sum_{y\not\in B}\Psi_{z,y}
		.\end{eqnarray*}
	Hence 
	\[
	\max\sett{\betrag{h(x)}}{x\not\in B}
	\leq
	\max\sett{\betrag{h(x)}}{x\not\in B}\sum_{y\not\in B}\Psi_{z,y}
	.\]
	By \eqref{bedmc} we on the other hand have 
	\begin{eqnarray}\label{istkleinereins}
	\sum_{y\not\in B}\Psi_{z,y}=\alpha(z)\sum_{y\not\in B}\Pi_{z,y}<1
	\text{ for all } z\not\in B
	.\end{eqnarray}
	We obtain \[0=\max\sett{\betrag{h(x)}}{x\not\in B}\sum_{y\not\in B}\Psi_{z,y}.\]
	Hence $h\equiv0$, proving the claim.
\end{proof}

\section{A Numerical Experiment}\label{sec:numerics}
To start with a toy example, consider $\set{0,...,20}\times\set{0,...,20}$ as the state space for the chain and a 
payoff function $g$ with $g(z,n)=\alpha^n f(z)$ with $f((5,5))=10$, $f((5,15))=f((15,15))=0$, $f((x,y))=5$ otherwise, see Figure \ref{fig:payoff}. 
The transition probabilities are given in the form
$$ p(x+1,y\mid x,y) = 0.5 \cdot p_x, \;  p(x-1,y\mid x,y) = 0.5 \cdot ( 1-p_x),$$
$$ p(x,y+1\mid x,y) = 0.5 \cdot p_y, \;  p(x,y-1\mid x,y) = 0.5 \cdot ( 1-p_y),$$
with reflecting boundaries. 
\begin{figure}
		\includegraphics{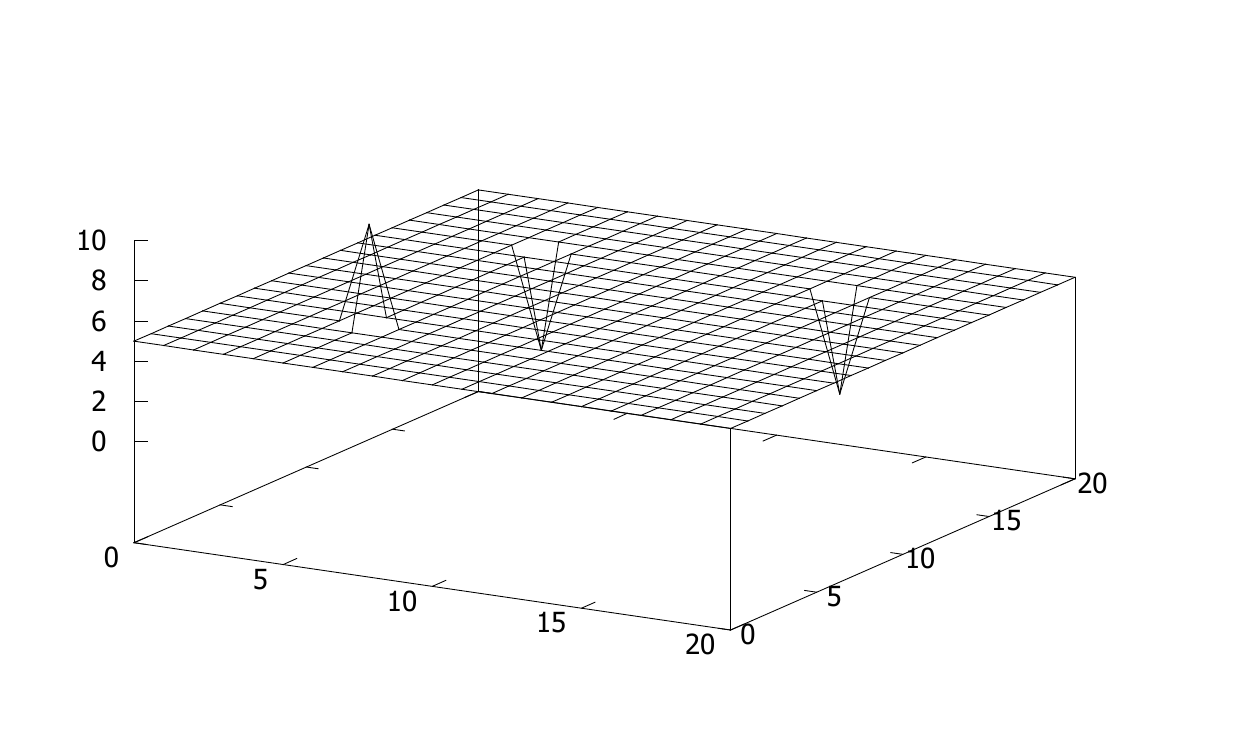}
		\caption{Undiscounted Pay-Offs in the toy example}\label{fig:payoff}
\end{figure}

With a discount-factor $\alpha=0.98^{1/20}$ our program recommends a stopping rule as given in Figure \ref{fig:stopping_set}
with corresponding optimal value shown in\ Figure \ref{fig:;value}.
\begin{figure}
	\begin{center}
		\includegraphics{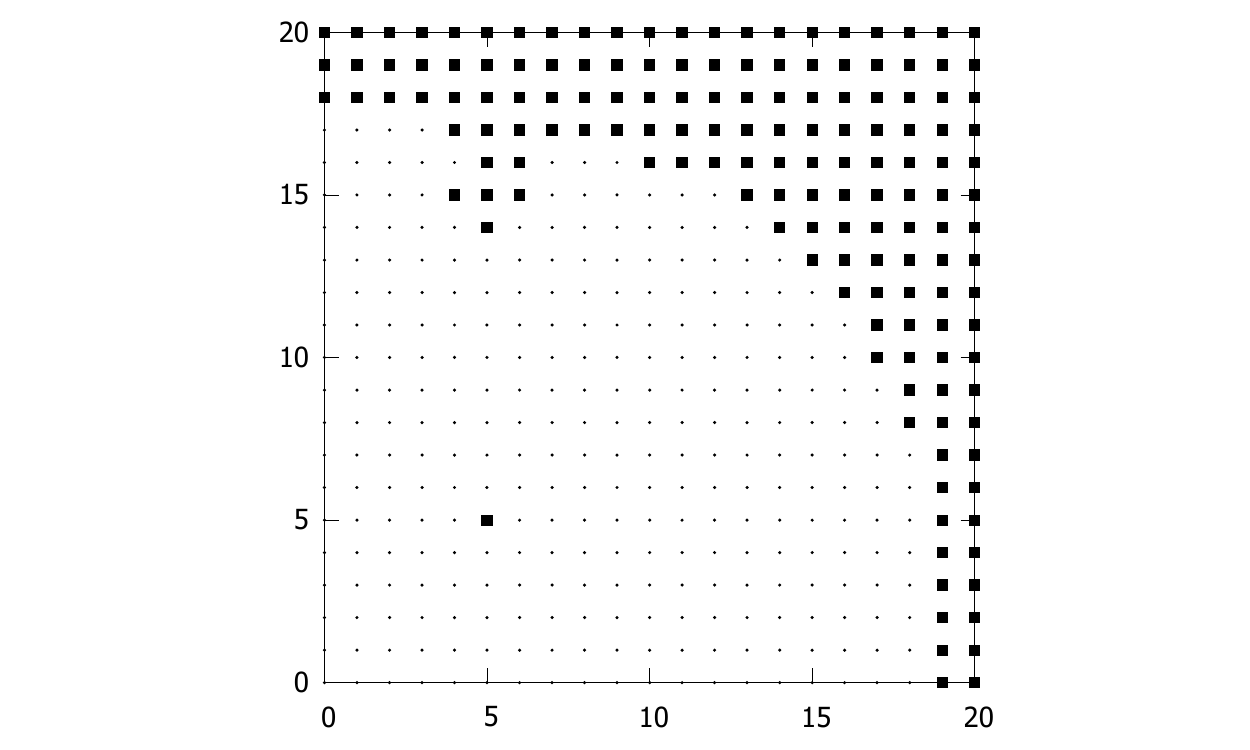}%
		\caption{Stopping-Points (big squares) of the optimal stopping rule in the toy example}\label{fig:stopping_set}
	\end{center}
\end{figure}
\begin{figure}
	\begin{center}
		\includegraphics{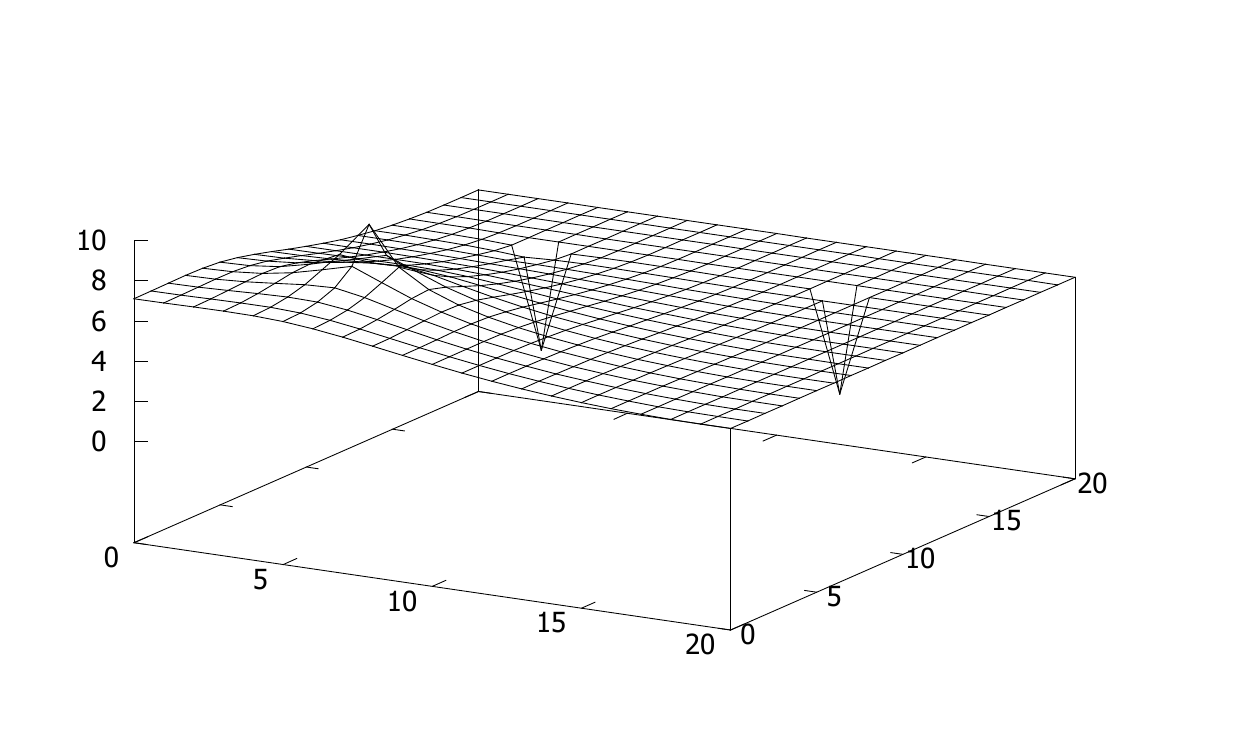}%
		\caption{Expected Values in the toy example}\label{fig:;value}
	\end{center}
\end{figure}
For constant $\kappa$ equivalent to $\set 1$ (the traditional forward improvement iteration algorithm) the program is solving the problem above within 19 iterations. The speed of solving the toy example above is so high that the effect of changing $\kappa$ cannot be examined therein. 

Thus we will increase the resolution by factor 10 in each dimension to study the effect of $\kappa$, hence we study a state space of $\set{0,...,200}\times\set{0,...,200}$ with pay-off $10$ at $(50,50)$, pay-off $0$ at $(50,150)$ and $(150,150)$ and 5 otherwise. We set $\alpha=0.9999$, leading to an optimal stopping rule of stopping at the four surrounding points of $(50,150)$ and $(150,150)$ as well as at $(50,50)$ and continuing at all other points.


At first we will examine $\kappa\equiv k$ for different $k$. We have to take into account two effects: With increasing $k$, more powers of the transition matrix have to be calculated while less iterations are needed.
In this case, a choice of $k=5$  seems to be a good balance, see Figure \ref{fig:runtime}.


%
%

\begin{figure}
	\includegraphics{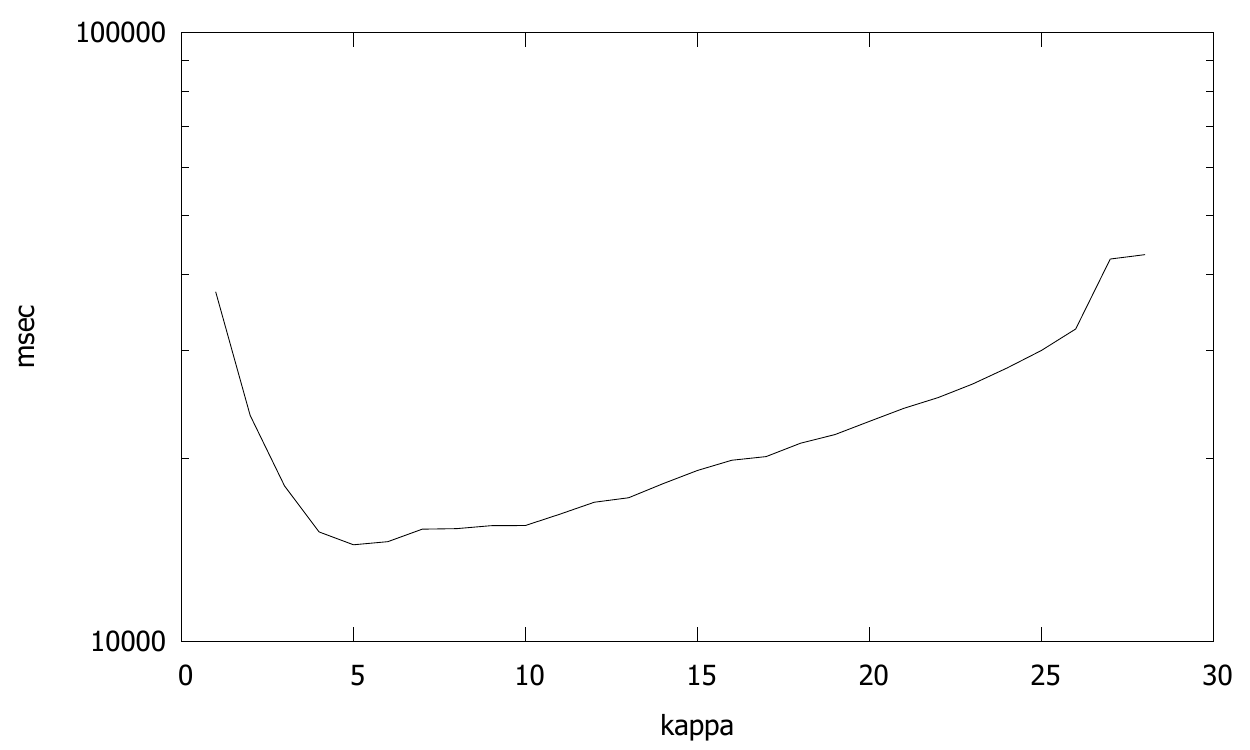}%
	\caption{Runtime of the algorithm (in log-scale) plotted against $\kappa\equiv k$.}\label{fig:runtime}
\end{figure}
%



\appendix
\section{Some Lemmata}
\begin{lemma}\label{PartialImproving}
	Assume $B\subseteq S$, $D\subseteq\N$, \fabist{0}{n}{<}{\infty}{}, $j\in D$, 
	$A\in\mathcal{A}_n$ with $A\subseteq\set{Z_n\in B^{*D_{<j}}\setminus B^{*D_{\leq j}}}$, $z\in S$.
	\\Then we have
	\[
	\integrala{A}{\XX_n}\leq\integrala{A}{\XX_{\tau_{n+j}(B)}}
	.\]
\end{lemma}

\begin{proof}
	For all $E\subseteq D$, $j\in E$  we have
	\begin{eqnarray*}
	B^{*E}
	&\subseteq&\nonumber
	B^{*E\setminus\set j},
	\\	B^{*E}
	&=&\nonumber
	\sett{z\in B^{*E\setminus\set j}}{g(z)\geq \CondExp{\XX_{\tau_{j}(B)}}{Z_0=z}},
	\\	B^{*E\setminus\set j}\setminus B^{*E}
	&=&
	\sett{z\in B^{*E\setminus\set j}}{g(z)< \CondExp{\XX_{\tau_{j}(B)}}{Z_0=z}}
	.\end{eqnarray*}
	
	Assume $(Z_m')_{m\in\NN}$ being an independent copy of $(Z_m)_{m\in\NN}$ and 
	\[\text{
		$\tau_p'(B)=\inf\sett{q\geq p}{Z_q'\in B}$ \fabist{0}{p}{<}{\infty}{for all}.
	}\]
	Then we have
	\begin{eqnarray*}
		&&		\integrala{A}{\XX_n}
\\&=&
		\integrala{A}{\Klammern{\prod_{i=0}^{n-1}\alpha(Z_i)}g(Z_n)}
\\&\leq&
		\integrala{A}{\Klammern{\prod_{i=0}^{n-1}\alpha(Z_i)}
			\CondExp{g\Klammern{Z'_{\tau'_j(B)}}\prod_{i=0}^{\tau'_{j}(B)-1}\alpha(Z'_i)}{Z'_0=Z_n}}
\\&=&
		\integrala{A}{\Klammern{\prod_{i=0}^{n-1}\alpha(Z_i)}
			\CondExp{g\Klammern{Z_{\tau_{n+j}(B)}}\prod_{i=n}^{\tau_{n+j}(B)-1}\alpha(Z_i)}{\mathcal{A}_n}}
\\&=&
		\integrala{A}{\Klammern{\prod_{i=0}^{n-1}\alpha(Z_i)}
			g\Klammern{Z_{\tau_{n+j}(B)}}\prod_{i=n}^{\tau_{n+j}(B)-1}\alpha(Z_i)}
		\\	&=&
		\integrala{A}{\Klammern{\prod_{i=0}^{\tau_{n+j}(B)-1}\alpha(Z_i)}g\Klammern{Z_{\tau_{n+j}(B)}}}
\\&=&
		\integrala{A}{\XX_{\tau_{n+j}(B)}}
		.\end{eqnarray*}
\end{proof}

\begin{lemma}
	\label{PartialImprovingII}
	Assume $B\subseteq S$, $D\subseteq\N$, $s,t\in\N$ with $s\leq t$ and $t-s\in D$, $A\in\mathcal{A}_s$ with $A\subseteq\set{Z_s\in B^{*D}}$. 
	Then we have
	\begin{eqnarray*}
	\integrala{A}{\XX_{\tau_t(B)}}
	\leq
	\integrala{A}{\XX_s}
	.\end{eqnarray*}
\end{lemma}	

\begin{proof}
	Assume $(Z_n')_{n\in\NN}$ being an independent copy of $(Z_n)_{n\in\NN}$ and 
	\[\text{$\tau_p'(B)=\inf\sett{q\geq p}{Z_q'\in B}$ \fabist{0}{p}{<}{\infty}{for all}.}\]
	Then we have
	\begin{eqnarray*}
		&&
		\integrala{A}{\XX_{\tau_t(B)}}
\\&=&
		\integrala{A}{\Klammern{\prod_{i=0}^{\tau_t(B)-1}\alpha(Z_i)}g\Klammern{Z_{\tau_t(B)}}}
		\\	&=&
		\integrala{A}{\Klammern{\prod_{i=0}^{s-1}\alpha(Z_i)}
			g\Klammern{Z_{\tau_t(B)}}\prod_{i=s}^{\tau_t(B)-1}\alpha(Z_i)}
\\&=&
		\integrala{A}{\Klammern{\prod_{i=0}^{s-1}\alpha(Z_i)}
			\CondExp{g\Klammern{Z_{\tau_t(B)}}\prod_{i=s}^{\tau_t(B)-1}\alpha(Z_i)}{\mathcal{A}_s}}
\\&=&
		\integrala{A}{\Klammern{\prod_{i=0}^{s-1}\alpha(Z_i)}
			\CondExp{g\Klammern{Z'_{\tau'_{t-s}(B)}}\prod_{i=0}^{\tau'_{t-s}(B)-1}\alpha(Z'_i)}{Z'_0=Z_s}}
		\\	&\omu{$Z_s\in B^{*D}$}{\leq}{$t-s\in D$}&
		\integrala{A}{\Klammern{\prod_{i=0}^{s-1}\alpha(Z_i)}g(Z_s)}
\\&=&
		\integrala{A}{\XX_s}
		.\end{eqnarray*}
\end{proof}

\bibliographystyle{plainnat}
\bibliography{FII}

\begin{thebibliography}{20}
\providecommand{\natexlab}[1]{#1}
\providecommand{\url}[1]{\texttt{#1}}
\expandafter\ifx\csname urlstyle\endcsname\relax
  \providecommand{\doi}[1]{doi: #1}\else
  \providecommand{\doi}{doi: \begingroup \urlstyle{rm}\Url}\fi

\bibitem[Asmussen(2003)]{Asm03}
S{\o}ren Asmussen.
\newblock \emph{Applied Probability and Queues}.
\newblock Applications of mathematics 51. Springer, 2003.
\newblock ISBN 0387002111.

\bibitem[Bender and Schoenmakers(2006)]{BS06}
Christian Bender and John Schoenmakers.
\newblock An iterative method for multiple stopping: convergence and stability.
\newblock \emph{Advances in Applied Probability}, 38\penalty0 (3):\penalty0
  729--749, 2006.
\newblock ISSN 0001-8678.

\bibitem[Cho and Stockbridge(2002)]{cho2002linear}
Moon~Jung Cho and Richard~H Stockbridge.
\newblock Linear programming formulation for optimal stopping problems.
\newblock \emph{SIAM Journal on Control and Optimization}, 40\penalty0
  (6):\penalty0 1965--1982, 2002.

\bibitem[Chow et~al.(1971)Chow, Robbins, and Siegmund]{CRS1971}
Yuan~Shih Chow, Herbert Robbins, and David Siegmund.
\newblock \emph{Great expectations: the theory of optimal stopping}.
\newblock Houghton Mifflin Company Boston, 1971.
\newblock ISBN 0395053145, 9780395053140.

\bibitem[Christensen(2014)]{christensen2014method}
S{\"o}ren Christensen.
\newblock A {M}ethod for {P}ricing {A}merican {O}ptions {U}sing
  {S}emi-{I}nfinite {L}inear {P}rogramming.
\newblock \emph{Mathematical Finance}, 24\penalty0 (1):\penalty0 156--172,
  2014.

\bibitem[Christensen and Irle(2020)]{CI19}
S{\"o}ren Christensen and Albrecht Irle.
\newblock The monotone case approach for the solution of certain
  multidimensional optimal stopping problems.
\newblock \emph{Stochastic Processes and their Applications}, 130\penalty0
  (4):\penalty0 1972--1993, 2020.

\bibitem[Christensen and Lindensj{\"o}(2018)]{christensen2018finding}
S{\"o}ren Christensen and Kristoffer Lindensj{\"o}.
\newblock On finding equilibrium stopping times for time-inconsistent markovian
  problems.
\newblock \emph{SIAM Journal on Control and Optimization}, 56\penalty0
  (6):\penalty0 4228--4255, 2018.

\bibitem[Christensen and Sohr(2020)]{christensen2020solution}
S{\"o}ren Christensen and Tobias Sohr.
\newblock A solution technique for l{\'e}vy driven long term average impulse
  control problems.
\newblock \emph{Stochastic Processes and their Applications}, 130\penalty0
  (12):\penalty0 7303--7337, 2020.

\bibitem[Ferguson(2008)]{Ferguson}
Thomas~S. Ferguson.
\newblock Optimal stopping and applications.
\newblock {http://www.math.ucla.edu/~tom/Stopping/Contents.html}, 2008.

\bibitem[Glasserman(2003)]{Gla03}
Paul Glasserman.
\newblock \emph{Monte Carlo Methods in Financial Engineering (Stochastic
  Modelling and Applied Probability)}.
\newblock Springer, 2003.
\newblock ISBN 0387004513.

\bibitem[Howard(1960)]{howard:dynamic1960}
Ronald~A. Howard.
\newblock \emph{Dynamic Programming and {M}arkov Processes}.
\newblock MIT Press, Cambridge, MA, 1960.

\bibitem[Irle(1980)]{Irle1980}
Albrecht Irle.
\newblock On the best choice problem with random population size.
\newblock \emph{Zeitschrift f\"ur Operations Research. Serie A. Serie B},
  24\penalty0 (5):\penalty0 177--190, 1980.
\newblock ISSN 0340-9422.

\bibitem[Irle(2006)]{Irl06}
Albrecht Irle.
\newblock A forward algorithm for solving optimal stopping problems.
\newblock \emph{Journal of Applied Probability}, 43\penalty0 (1):\penalty0
  102--113, 2006.
\newblock ISSN 0021-9002.

\bibitem[Irle(2009)]{Irl09}
Albrecht Irle.
\newblock On forward improvement iteration for stopping problems.
\newblock In \emph{Proceeding of the International Workshop of Sequential
  Methodologies}, Troyes, 2009.

\bibitem[Kolodko and Schoenmakers(2006)]{KS06}
Anastasia Kolodko and John Schoenmakers.
\newblock Iterative construction of the optimal {B}ermudan stopping time.
\newblock \emph{Finance and Stochastics}, 10\penalty0 (1):\penalty0 27--49,
  2006.
\newblock ISSN 0949-2984.

\bibitem[Miclo and Villeneuve(2021)]{miclo_villeneuve_2021}
Laurent Miclo and St{\'e}phane Villeneuve.
\newblock On the forward algorithm for stopping problems on continuous-time
  markov chains.
\newblock \emph{Journal of Applied Probability}, 58\penalty0 (4):\penalty0
  1043--1063, 2021.
\newblock \doi{10.1017/jpr.2021.11}.

\bibitem[Presman(2011)]{presman2011new}
Ernst Presman.
\newblock A new approach to the solution of optimal stopping problem in a
  discrete time.
\newblock \emph{Stochastics An International Journal of Probability and
  Stochastic Processes}, 83\penalty0 (4-6):\penalty0 467--475, 2011.

\bibitem[Shiryaev(2008)]{MR2374974}
Albert~N. Shiryaev.
\newblock \emph{Optimal stopping rules}, volume~8 of \emph{Stochastic Modelling
  and Applied Probability}.
\newblock Springer-Verlag, Berlin, 2008.
\newblock ISBN 978-3-540-74010-0.
\newblock URL \url{https://mathscinet.ams.org/mathscinet-getitem?mr=2374974}.
\newblock Translated from the 1976 Russian second edition by A. B. Aries,
  Reprint of the 1978 translation.

\bibitem[Sonin(1999)]{sonin1999elimination}
Isaac Sonin.
\newblock The elimination algorithm for the problem of optimal stopping.
\newblock \emph{Mathematical methods of operations research}, 49\penalty0
  (1):\penalty0 111--123, 1999.

\bibitem[Tsitsiklis and Van~Roy(1999)]{tsitsiklis1999optimal}
John~N Tsitsiklis and Benjamin Van~Roy.
\newblock Optimal stopping of markov processes: Hilbert space theory,
  approximation algorithms, and an application to pricing high-dimensional
  financial derivatives.
\newblock \emph{IEEE Transactions on Automatic Control}, 44\penalty0
  (10):\penalty0 1840--1851, 1999.

\end{thebibliography}

\end{document}